%% file: main.tex
\documentclass[reqno,a4paper,11pt]{amsart}
\input{header.tex}

\usepackage{tikz}
\usepackage{tkz-euclide}
\usetikzlibrary{
    angles,
    quotes,
    calc,
    intersections,
    decorations.pathreplacing
}

\usepackage{geometry}
\usepackage{thm-restate}

\newcommand{\zetavalue}{2^{-20} d^{-12}}
\newcommand{\epsvalue}{2^{-200} d^{-112}}
\newcommand{\alphavalue}{2^{-5} d^{-3}}
\DeclarePairedDelimiterX{\sqnorm}[1]
  {\lVert}{\rVert^2}{\ifblank{#1}{\,\cdot\,}{#1}}

\usepackage[normalem]{ulem}

\DeclareMathOperator{\Conv}{Conv}
\newcommand{\setmp}{\set{-1,+1}}
\newcommand{\intmp}{[-1,+1]}

\begin{document}

\title{Parity in vector balancing and reverse Littlewood--Offord problems}

\author{Lawrence Hollom}
\address{Institute of Mathematics, EPFL, CH-1015 Lausanne, Switzerland}
\email{lawrence.hollom@epfl.ch}

\author{Gregory B. Sorkin}
\address{Department of Mathematics, The London School of Economics and Political Science, Houghton Street, WC2A 2AE, United Kingdom}
\email{g.b.sorkin@lse.ac.uk}

\begin{abstract}
  We consider the probability that the random signed sum $\xi_1 v_1 + \dotsb + \xi_n v_n$ lies within a given distance $r$ of the origin, where $v_1,\dotsc,v_n \in \mathbb{R}^d$ are fixed unit vectors and $\xi_1,\dotsc,\xi_n$ are independently and uniformly distributed on $\{-1,+1\}$.
  We demonstrate that, for certain values of $r$, the infimum of this probability is sensitive to the parity of $n$.

  We prove that, for any $d\geq 3$, there is some $\varepsilon = \varepsilon(d) > 0$ such that for any $n \not\equiv d \mod 2$ and unit vectors $v_1,\dotsc,v_n\in \mathbb{R}^d$, there are signs $\eta_1,\dotsc,\eta_n \in \{-1,+1\}$ such that $\|\sum_{i=1}^n \eta_i v_i\| \leq \sqrt{d - \varepsilon}$. 
  In other words, $\mathbb{P}(\| \xi_1 v_1 + \dotsb + \xi_n v_n \| \leq \sqrt{d-\varepsilon}) > 0$.
  This is in contrast to the case of $n\equiv d \mod 2$, in which this probability can be zero.
  In the case of $d=2$ and odd $n$, we give a construction demonstrating that $\mathbb{P}(\|\xi_1 v_1 + \dotsb + \xi_n v_n\| \leq 1)$ can decay exponentially as $n$ increases, improving on the previous best-known bound of $O(n^{-3/2})$.
\end{abstract}

\maketitle


\section{Introduction}
\label{sec:intro}

Given unit vectors $v_1,\dotsc,v_n \in \RR^d$ and independent random variables $\xi_1,\dotsc,\xi_n$, each taking values $-1$ and $+1$ with probability $1/2$, we are interested in the probability that the random signed sum $\xi_1 v_1 + \dotsb + \xi_n v_n$ lies in a given set $S\sseq \RR^d$.
This problem has been studied in various forms since the 1943 paper of Littlewood and Offord \cite{Littlewood1943-ax}.
While these ``Littlewood--Offord problems'' often concern the maximum of the probability that $\xi_1 v_1 + \dotsb + \xi_n v_n \in S$, we are concerned with the so-called \emph{reverse Littlewood--Offord} problem, in which the object of study is the minimum of the above probability.

A central result in this direction is the following theorem of Beck \cite{Beck1983-ef}.

\begin{theorem}[Beck]
\label{thm:beck}
    For any $d\geq 1$, there is a constant $c_d > 0$ such that the following holds.
    If $v_1,\dotsc, v_n \in \RR^d$ have $\norm{v_i} \leq 1$ for each $1\leq i \leq n$, and if $\xi_1,\dotsc,\xi_n$ are independent Rademacher random variables, then
    \begin{equation*}
        \prob{\norm{\xi_1 v_1 + \dotsb + \xi_n v_n} \leq \sqrt{d}} \geq \frac{c_d}{n^{d/2}}.
    \end{equation*}
\end{theorem}

In the above theorem, and throughout the paper, $\norm{}$ denotes the Euclidean $\l_2$-norm.
The $n^{-d/2}$ behaviour is exactly what would be expected for a non-degenerate symmetric $d$-dimensional Gaussian random variable whose covariance grows linearly with $n$.

While both the value $\sqrt{d}$ of the radius and the order $n^{-d/2}$ in \Cref{thm:beck} can be seen to be tight by taking $v_1,\dotsc,v_n$ to be copies of the vectors in an orthonormal basis $e_1,\dotsc,e_d$ with an odd number of copies of each $e_i$, recent work \cite{HJNS24,HPS26} has highlighted that the parity of $n$ also plays a central role.
Indeed, the above construction only shows that \Cref{thm:beck} is tight when $n\equiv d\mod{2}$, and it is natural to ask whether this is a genuine phenomenon or merely an artefact of the above construction.
We show that the parity of $n$ is significant in all dimensions $d$, and give the first improvement on Beck's $\sqrt{d}$ bound under the relevant parity condition.

\begin{theorem}
\label{thm:balancing}
    For every integer $d\geq 3$, there is $\eps = \eps(d) > 0$ such that, for any sequence $v_1,\dotsc,v_n \in \RR^d$ of unit vectors with $n \not\equiv d\mod 2$, there are signs $\eta_1,\dotsc,\eta_n \in\setmp$ such that 
    \begin{equation}
    \label{eq:balancing}
        \norm[\Big]{\sum_{i=1}^n \eta_i v_i} \leq \sqrt{d-\eps}.
    \end{equation}
\end{theorem}

In particular, we may take $\eps = \epsvalue$ in \Cref{thm:balancing}.
However, this value of $\eps$ is surely far from optimal; \Cref{thm:balancing} is not the end of the story, and we suspect that it should be possible to take $\eps=1$.
To properly put this in context, we now give a more thorough history of this problem.
Throughout the paper, we use the convention that $\xi_1,\dotsc,\xi_n \in\setmp$ are independent random signs, and $\eta_1,\dotsc,\eta_n \in\setmp$ are deterministic signs.

The starting point for the above problems is the following 1945 conjecture of Erd\H{o}s \cite{Erdos1945-fu}.

\begin{conjecture}[Erd\H{o}s]
\label{conj:erdos}
    There is a constant $c>0$ such that, for any integer $n$ and any unit vectors $v_1,\dotsc,v_n\in \RR^2$, if $\xi_1,\dotsc,\xi_n$ are distributed independently and uniformly at random on $\setmp$, then
    \begin{equation*}
        \prob[\big]{\norm{\xi_1 v_1 + \dotsb + \xi_n v_n} \leq 1} \geq \frac{c}{n}.
    \end{equation*}
\end{conjecture}

\Cref{conj:erdos} can be seen to be incorrect as stated for even $n$ by taking $k_1$ copies of $(1,0)$ and $k_2$ copies of $(0,1)$, provided that $k_1$ and $k_2$ are both odd.
This observation was attributed to Erd\H{o}s, S\'{a}rk\"{o}zy, and Szemer\'{e}di by Beck~\cite{Beck1983-ef}, and was later also made by Carnielli and Carolino \cite{Carnielli2011-mq}.
One simple way to fix \Cref{conj:erdos} is to instead ask for the probability that the norm of the signed sum is at most $\sqrt{2}$; this naturally generalises to $\sqrt{d}$ in $d$ dimensions, and we arrive back at \Cref{thm:beck}.

More recently, He, Ju\v{s}kevi\v{c}ius, Narayanan, and Spiro~\cite{HJNS24} reproved Beck's result for $d=2$ and noted that the parity of $n$ seems to play an important role.
In particular, they conjectured that Erd\H{o}s' conjecture should hold if one conditions on $n$ being odd.
Some evidence for this conjecture is a result of Swanepoel~\cite[Theorem A]{Swanepoel2000-ha}, later reproved by B\'{a}r\'{a}ny, Ginzburg, and Grinberg~\cite[Theorem 1]{Barany2013-vn}, which implies that $\prob{\norm{\xi_1 v_1 + \dotsb + \xi_n v_n} \leq 1}$ is strictly positive when $n$ is odd (in contrast to when $n$ is even, when this probability can be zero).

However, in \cite{HPS26} this conjecture was disproved, by means of constructing vectors $v_1,\dotsc,v_n$ with $\prob{\norm{\xi_1 v_1 + \dotsb + \xi_n v_n} \leq 1} = O(n^{-3/2})$.
Moreover, the problem of determining the minimum value taken by the above probability for odd $n$ was left open.
Here, we show that this bound can in fact be exponentially small.

\begin{theorem}
\label{thm:construction}
    Fix $c=1/20$ and an odd integer $n$.
    Define $v_n = (1,0)$ and, for $1\leq i \leq \floor{n/2}$, set $v_{2i-1} = v_{2i} = (\cos \theta_i, \sin \theta_i)$, where $\theta_i = \arcsin c^i$.
    Then, where $\xi_1,\dotsc,\xi_n$ are independent Rademacher random variables, 
    \begin{equation*}
        \prob[\big]{\norm{\xi_1 v_1 + \dotsb + \xi_n v_n} \leq 1} = 2^{-\floor{n/2}}.
    \end{equation*}
\end{theorem}

Combining \Cref{thm:construction} with the results of \cite{HPS26} yields the following, which is the state of the art in two dimensions.

\begin{table}[ht]
    \begin{tabular}{c|c}
        $r$ & $\inf_{v_1,\dotsc,v_n\in\SS^{1}} \Prob \p[\big]{\norm{\sum_{i=1}^n \xi_i v_i} \leq r}$ \\
        \hline
        $<1$ & $0$ \\
        $=1$ & $\Omega(0.525^n)$ and $O((1/\sqrt{2})^n)$ \\
        $>1$ & $\Theta_r(n^{-1})$
    \end{tabular}
    \vspace{0.5em}
    \caption{Best known bounds in two dimensions when $n$ is odd. The upper bound at $r=1$ is our \Cref{thm:construction}; all other results are from \cite{HPS26}.}
    \label{table}
\end{table}

In particular, the exponential behaviour of the above probability at $r=1$ -- in contrast to the polynomial behaviour for $r>1$ and zero for $r<1$ -- highlights the titular ``double-jump phase transition'' of \cite{HPS26}.

Note that, if $e_1,\dotsc,e_d$ is an orthonormal basis of $\RR^d$, and the sequence $v_1,\dotsc,v_n$ consists of an odd number of copies of each of $e_1,\dotsc,e_{d-1}$ and an even number of copies of $e_d$ (so $n\not\equiv d\mod{2}$), we find that
\begin{equation*}
    \min_{\eta_1,\dotsc,\eta_n \in \setmp} \norm[\Big]{\sum_{i=1}^n \eta_i v_i} = \sqrt{d-1}.
\end{equation*}
Note moreover that, in the case of $d=2$, we have $\sqrt{d-1} = 1$.
In this light, both of our results -- \Cref{thm:balancing,thm:construction} -- can be seen as steps towards 
resolving the following conjecture.

\begin{restatable}[\cite{HPS26}]{conjecture}{balanceQuestionMain}
\label{q:balancing}
    Let $v_1, \dotsc, v_n \in \RR^d$ be unit vectors with $n \not\equiv d \mod{2}$.
    Then there are signs $\eta_1, \dotsc, \eta_{n} \in \setmp$ with
    \begin{equation*}
        \norm[\Big]{\sum_{i=1}^{n} \eta_i v_i} \leq \sqrt{d-1} \;\; . 
    \end{equation*}
\end{restatable}
The conjecture holds for $d=3$, $n=4$: this can be viewed as an optimisation problem over 5 real variables, regions of this 5-dimensional space can be excluded, and \textit{ad hoc} reasoning applied to the remaining cases.

Such questions asking for only a single choice of signs are referred to as \emph{vector balancing} problems (see e.g.\ \cite{Swanepoel2000-ha}).
If the answer to this question is positive, the natural next step would be to understand the probability
\begin{equation}
\label{eq:high-dim-double-jump}
    \prob[\bigg]{\norm[\Big]{\sum_{i=1}^{n} \xi_i v_i} \leq \sqrt{d-1}}
\end{equation}
when $n\not\equiv d \mod{2}$.
In particular, it is of great interest whether there is a double-jump phase transition matching that which occurs for $d=2$, as shown in the $r=1$ row of \Cref{table}.

In this context, \Cref{thm:balancing} shows that the parity of $n$ is significant, even if it does not go all the way to the $\sqrt{d-1}$ threshold posited by \Cref{q:balancing}. 
While it is known that, in two dimensions, the probability \eqref{eq:high-dim-double-jump} is positive \cite{Barany2013-vn,Swanepoel2000-ha}, \Cref{thm:construction} shows that this probability can be exponentially small.

We also note the wealth of other open problems concerning the random signed sum $\xi_1 v_1 + \dotsb + \xi_n v_n$.
Two prominent examples of such problems are the Beck--Fiala conjecture \cite{beck1981integer} and the Koml\'{o}s conjecture (see \cite{lovasz1986discrepancy,spencer1985six}), which concern bounding the $\l_\infty$-norm of the above random sum under certain assumptions on the $v_i$. 
Both of these problems have received significant recent study; see, for example, the recent work of Bansal and Jiang \cite{BJ26}.

We now provide an outline of the paper along with brief sketches of the proofs.

\subsection{AI use and formal verification}

All ideas in this paper were generated and the proofs written by the authors. AI (ChatGPT-5.6 Sol and ChatGPT-6 Astra) was used only in proofreading the manuscript and checking for errors. 

Once the present version of the manuscript was written, the main results (\Cref{thm:balancing,thm:construction}) were formally verified in Lean 4 by ChatGPT-6 Astra, and the code can be found in a public GitHub repository of the first author.\footnote{\url{https://github.com/LawrenceHollom/lean-vector-balancing}}

\subsection{Paper outline and notation}

In \Cref{sec:construction} we provide a proof of \Cref{thm:construction}, showing that if $\xi_1 v_1 + \dotsb + \xi_n v_n$ lies in the unit ball for vectors $v_i$ as given in the statement of the theorem, then we must have $\xi_{2i-1} = - \xi_{2i}$ for all $i < n/2$.

We then state a few preliminary results in \Cref{sec:preliminaries}, which will find use in our proofs throughout the rest of the paper.
In particular, we will use some standard results about covering zonotopes with parallelotopes. 
The following definition serves as a reminder and establishes notation.
\begin{definition}
    For a sequence $V=(v_1,\dotsc,v_n)$ of vectors, let 
    \begin{equation*}
        S(V)\defined \set[\Big]{\sum_{i=1}^n \eta_i v_i \st \eta_i \in \set{-1, +1}}
    \end{equation*} 
    be the set of signed sums of the vectors $V$. The continuous version of this is the \emph{zonotope}
    \begin{equation*}
        Z(V)\defined \set[\Big]{\sum_{i=1}^n \lambda_i v_i \st \lambda_i \in [-1, +1]}.
    \end{equation*} 
    If $V$ is linearly independent then $Z(V)$ is called a \emph{parallelotope}.
\end{definition}
With a slight abuse of notation, we will also consider a sequence $V$ of vectors as a multiset, and thus use notation such as $W\sseq V$ and $V \setminus W$, defined entirely as would be expected.

\begin{definition} \label{def:approx}
    A sequence $V = (v_1,\dotsc,v_n)$ of vectors in $\RR^d$ is said to be \emph{$r$-approximating} for some $r > 0$ if, for all $\lambda_1,\dotsc,\lambda_n \in \intmp$, there are signs $\eta_1,\dotsc,\eta_n\in\setmp$ such that
    \begin{equation} \label{eq:approx}
        \norm[\Big]{\sum_{i=1}^n (\eta_i - \lambda_i) v_i}^2 \leq r.
    \end{equation}
\end{definition}
In other words, $V$ is $r$-approximating if every point in $Z(V)$ 
is within squared distance $r$ of some point in $S(V)$.

\Cref{thm:beck} immediately implies that a set $V$ of sub-unit vectors can approximate 0 to within $\sqrt d$, but Beck generalises this and shows that all points in $Z(V)$ can be approximated this well, i.e., $V$ is $d$-approximating \cite{Beck1983-ef}.
\Cref{sec:preliminaries} among other things re-proves this result. 
The proof has two components.
One, \Cref{lem:beck-elimination}, taken directly from Beck, reduces the general case to that of $n=d$. 
In the process of that reduction, the point to be approximated shifts within $Z(V)$,
so Beck's generalisation is necessary for the proof. 
The second is to show that $d$-approximation is possible when $n=d$.
Beck does so with a probabilistic argument for the existence of signs giving such an approximation;
in \Cref{lem:stability} we do so simply by greedily constructing such signs, giving some sharper result that we will need to establish \Cref{thm:balancing}.

The rest of the paper is dedicated to the proof of \Cref{thm:balancing}.
Note that it can no longer suffice to rely on Beck's reduction argument.
For one thing, it is not true that all points in $Z(V)$ can be better than $\sqrt d$-approximated.
(If $V$ consists of an orthonormal basis with $v_1$ repeated as $v_{d+1}$, the best possible approximation to $v_1$, which of course lies in $Z(V)$, is $\sqrt d$.)
Also, the reduction argument, at least as usually applied, changes $n$ to $n-1$ and thus does not preserve the parity condition $n \not\equiv d \mod{2}$.

Instead, we reason through a dichotomy. 
\begin{definition}
Say that vectors $u,w$ are \emph{$\alpha$-oblique} (or simply `oblique') if $\abs{\inner{u}{w}} \in (\alpha, 1 - \alpha)$.
Otherwise, either $\abs{\inner u w} \leq \alpha$ in which case we say that $u$ and $w$ are $\alpha$-orthogonal (or just `nearly' orthogonal), or $\abs{\inner u w} \geq 1-\alpha$ and we say they are $\alpha$-parallel (or just `nearly' parallel).
\end{definition}
If $V$ contains no oblique pair (for a suitably chosen parameter), \Cref{lem:cluster} shows that the vectors form $d$ clusters that approximate an orthonormal basis (some clusters may be empty). 
If the clusters were exactly orthonormal, each consisting of repetitions of a single vector, then within each cluster we could cancel out pairs of vectors, leading to an approximation of 0 with squared length equal to the number of odd-cardinality clusters.
Because $n \not\equiv d \mod{2}$, so not all $d$ clusters can have odd cardinality, so this means 0 is approximated to within $\leq \sqrt{d-1}$.
\Cref{lem:clusterShort} shows that this holds, approximately, when $V$ is clustered.
Note that this argument covers all $n$ (with $n \not\equiv d \mod{2}$) but, of necessity, applies only to approximation of~0.

The other case is that $V$ contains some oblique pair $u, w$. 
In this case we ignore the parity condition, and deduce from the obliqueness alone that the vectors are $(d-\eps)$-approximating.
The elimination procedure implicit in \Cref{lem:beck-elimination} lets us reduce to the case $n=d+1$;
this case is established with the stability result of \Cref{lem:dichotomy}, proved in \Cref{sec:stability}.

Finally, we discuss a range of open problems and directions for future research in \Cref{sec:conclusion}.


\section{Exponentially small probability of a signed sum with norm at most 1}
\label{sec:construction}

In this section we prove \Cref{thm:construction}. \Cref{fig:construction} gives an approximate diagram of the construction.

\begin{figure}[ht]
    \centering
    \begin{tikzpicture}[scale=1.5]
        \tkzDefPoints{0/0/O, 1/0/R}
        \tkzDefPointsBy[rotation=center O angle 0](R){V1}
        \tkzDefPointsBy[rotation=center O angle 1.25](R){V8}
        \tkzDefPointsBy[rotation=center O angle 2.5](R){V7}
        \tkzDefPointsBy[rotation=center O angle 5](R){V2}
        \tkzDefPointsBy[rotation=center O angle 10](R){V3}
        \tkzDefPointsBy[rotation=center O angle 20](R){V4}
        \tkzDefPointsBy[rotation=center O angle 40](R){V5}
        \tkzDefPointsBy[rotation=center O angle 80](R){V6}
        
        \tkzDefPointsBy[rotation=center O angle 0](R){U1}
        \tkzDefPointsBy[rotation=center O angle 90](R){U2}
        \tkzDefPointsBy[rotation=center O angle 180](R){U3}
        \tkzDefPointsBy[rotation=center O angle 270](R){U4}
        
        \tkzDrawCircle[line width=1pt, gray](O,R)
        \tkzDrawLines[line width=0.5pt, gray, dashed, add=0.2 and 0.2](U1,U3)
        \tkzDrawLines[line width=0.5pt, gray, dashed, add=0.2 and 0.2](U2,U4)
        
        \tkzDrawPoints[size=2.5pt, black](V1,V2,V3,V4,V5,V6)
        \tkzLabelPoint[xshift=28pt,yshift=1pt](V1){$v_n = (1,0)$}
        \tkzLabelPoint[xshift=15pt,yshift=6pt](V2){$\dots$}
        \tkzLabelPoint[xshift=16pt,yshift=7pt](V4){$v_5,v_6$}
        \tkzLabelPoint[xshift=10pt,yshift=14pt](V5){$v_3,v_4$}
        \tkzLabelPoint[xshift=7pt,yshift=14pt](V6){$v_1,v_2$}
    \end{tikzpicture}
    \caption{An approximate diagram of the construction in \Cref{thm:construction}. The actual points chosen are much closer to $(1,0)$, and the angles decay to zero much faster than those shown.}
    \label{fig:construction}
\end{figure}
Throughout this section, $n$ is odd and $v_1,\dotsc,v_n$ are as defined in \Cref{thm:construction}.
That is, $v_n = (1,0)$, and $v_{2i-1} = v_{2i} = (\cos \theta_i, \sin\theta_i)$, where $\theta_i = \arcsin c^i$ for $c = 1/20$.
We must show that at most $2^{\ceil{n/2}}$ choices of the sign sequence $\eta_1,\dotsc,\eta_n$ result in the signed sum $\sum_{i=1}^n \eta_i v_i$ lying within distance 1 of the origin.
Indeed, it suffices to prove the following claim.

\begin{claim}
\label{claim:construction}
    If $\norm{\sum_{i=1}^n \eta_i v_i} \leq 1$, then, for all $1\leq i \leq \floor{n/2}$, we have $\eta_{2i-1} = - \eta_{2i}$.
\end{claim}

\begin{proof}
    Suppose that not all $\eta_{2i-1} = - \eta_{2i}$. Let $\eta_{2k-1} = \eta_{2k}$ be the first equal pair, and write $v_{2i-1} = v_{2i} = (x_i, y_i)$. 
    Each pair $v_{2i-1}, v_{2i}$ sums either to 0 (including for all pairs $i<k$) or to $2\eta_{2i}\p[\big]{\sqrt{1-y_i^2}, y_i}$. 
    
    The $y$ coordinate of the sum of all pairs and $v_n$ (contributing 0) thus has
    \begin{align*}
        \abs{y} &= \abs[\Big]{2\eta_{2k}y_k + \sum_{i = k+1}^{(n-1)/2} (\eta_{2i-1} + \eta_{2i}) y_i}
         \geq 2\p[\Big]{y_k - \sum_{i > k} y_i}
         = 2 \p[\Big]{ c^k - \frac{c^{k+1}}{1-c} }
         = 2 c^k \cdot \frac{18}{19}  ,  
    \end{align*}
    using $c=1/20$.

    For the $x$ coordinate, write $\sqrt{1-y_i^2}$ as $1-\Delta_i$ and note that $\Delta_i \leq 0.51 y_i^2$ for $y_i \leq c = 1/20$.
    Let $I = \set{i \st \eta_{2i-1} = \eta_{2i}}$.
    The $x$ coordinate of the sum of all pairs and $\eta_n v_n$ (contributing~$\eta_n = \pm 1$) is $x = \eta_n + 2\sum_{i\in I} \eta_{2i} (1-\Delta_i)$. Since $\sum_{i \geq 1}\Delta_i<1/4$,
    \begin{align*}
        \abs{x} & \geq 1 - 2\sum_{i\geq k} \Delta_i
        \geq 1- 2\sum_{i\geq k} 0.51 y_i^2
        = 1- 1.02 \frac{(c^k)^2}{1-c^2}
        \geq 1 - 1.03 c^{2k} .
    \end{align*}

    This gives sum vector $(x,y)$ with squared length $x^2+y^2>1$. Thus, the only way to achieve length 1 or less is to have $\eta_{2i-1} = - \eta_{2i}$ in every pair.
\end{proof}

With \Cref{claim:construction} proved, so too is \Cref{thm:construction}.


\section{Preliminary results}
\label{sec:preliminaries}

We now present the results we will make use of throughout the rest of the paper which derive primarily from other sources.

We denote the convex hull of a set $X$ as $\Conv(X)$.
The following is a standard fact about zonotopes, but we nevertheless provide a proof for completeness.

\begin{remark} \label{CSeqZ}
    For any sequence $V=(v_1,\dotsc,v_n)$ of vectors, 
    \begin{equation*}
      \Conv(S(V)) = Z(V) .
    \end{equation*}
\end{remark}

\begin{proof}
    It is clear that $\Conv(S(V)) \subseteq Z(V)$: in any combination in $\Conv(S(V))$, the coefficient of $v_i$ is a convex combination of individual coefficients $-1$ and $+1$, thus in $\intmp$.
    That $Z(V) \subseteq \Conv(S(V))$ can be shown by induction. First, $\lambda_1 v_1 + \lambda_2 v_2 + \dotsb + \lambda_n v_n$ is in $\Conv(S(V))$ if both $v_1 + \lambda_2 v_2 +\dotsb + \lambda_n v_n$ and $-v_1 + \lambda_2 v_2 +\dotsb+ \lambda_n v_n$ are. Then for each of these points, replace $\lambda_2$ with $\pm 1$, and so on.
\end{proof}

\Cref{lem:stability} addresses the case $n=d$. Its main assertion was proved by Beck \cite[Lemma 2.2]{Beck1983-ef} using a non-constructive probabilistic argument, randomizing the choices of the $\eta_i$ independently but as functions of $\lambda_i$.
We give a simple algorithmic proof whose essence is \Cref{lem:uv}.

In application of \Cref{lem:uv}, $u$ will be the error accrued by step $k$, $u = \sum_{i=1}^k (\eta_i - \lambda_i) v_i$, and $v=v_{k+1}$ will be the next vector, adding $\lambda v$ to the target and thus, for a given choice of $\eta = \eta_{k+1}$, changing the error to $u+(\eta-\lambda)v$.
The lemma allows control of the increase in the squared error length at each step.

\newcommand{\gps}{g_+}
\newcommand{\gms}{g_-}
\newcommand{\gsq}{g}
\newcommand{\sps}{s_+^2}
\newcommand{\sms}{s_-^2}
\newcommand{\ssq}{s^2}

\begin{lemma}
\label{lem:uv}
For any vectors $u$ and $v$ at angle $\theta$, and $\lambda \in \intmp$, 
let $s_\eta(u,v,\lambda) = \norm{u+(\eta-\lambda) v}$
and let $s = \min\set{s_\eta(u,v,\lambda) \colon \eta \in \setmp}$.
Then
\begin{align}
 \ssq &\leq \sqnorm u + \sqnorm v - \lambda^2 \sqnorm v     \label{eq:delta}
\intertext{
If $\lambda=0$, then}
 \ssq &\leq \sqnorm u + \sqnorm v - 2 \abs{\inner u v} \label{eq:lzero} .
\intertext{%
If $\norm u, \norm v \leq 1$ then}
 \ssq &\leq \sqnorm u + \sqnorm v - (\inner u v)^2  \label{eq:theta} .
\end{align}
\end{lemma}

\begin{proof}
Let $a = \norm u$ and $b=\norm v$.
Changing notation correspondingly, let
\begin{align} 
\sps := s_{+1}(a,b,\lambda,\theta)^2
  &= (a + (1-\lambda) b \cos(\theta))^2  + ((1-\lambda) b \sin(\theta))^2  \notag
  \\&= a^2 + (\lambda-1)^2 b^2 -2 a b (\lambda-1) \cos(\theta)  \label{eq:sp} ,
\\ 
\sms := s_{-1}(a,b,\lambda,\theta)^2 
 &= a^2 + (\lambda+1)^2 b^2 -2 a b (\lambda+1) \cos(\theta)  .  \label{eq:sm}
\end{align}

\textbf{To prove \eqref{eq:delta}}, we will show that the smaller of the gaps 
between the left and right sides of \eqref{eq:delta}, in the two cases, is at most~0.
That is, we will show that $\gsq := \min\set{\gps,\gms} \leq 0$ where
\begin{align} 
  \gps(a,b,\lambda,\theta) 
  & := \sps
  - \big( a^2 + b^2 - \lambda^2 b^2 \big)  
   = -2 b (\lambda-1) (a \cos(\theta) - b \lambda)  \label{eq:gp} ,
 \\ 
  \gms(a,b,\lambda,\theta) 
     & := \sms
   - \big( a^2 + b^2 - \lambda^2 b^2 \big)  
  = -2 b (\lambda+1) (a \cos(\theta) - b \lambda)  \label{eq:gm}. 
\end{align}

Specifically, we show that, for given $a$, $b$, and $\theta$, the maximum of $\gsq$ over all $\lambda \in \intmp$ is at most~0.
The maximum must occur for $\lambda$ an endpoint ($-1$ or $+1$), a local maximum of $\gps$ or $\gms$, or where $\gps=\gms$.

If $\lambda=1$ then $\gps=0$, while if $\lambda=-1$ then $\gms=0$;
in either case $\gsq \leq 0$.

There is no local maximum of $\gps$ nor $\gms$, since it is evident from \eqref{eq:gp} and \eqref{eq:gm} that both have $\lambda$ second derivative $4b^2 \geq 0$. 

The binding case is the one where $\gps = \gms$, i.e., 
\begin{align}
  0 = \gps - \gms  &= 4 b (a \cos(\theta) - b \lambda) ,  \label{eq:geq}
\end{align}
implying $\cos(\theta) = \frac b a \lambda$.
This makes \eqref{eq:gp} and \eqref{eq:gm} both zero,
proving \eqref{eq:delta}.

\textbf{To prove \eqref{eq:lzero}}, taking $\lambda=0$, with reference to \eqref{eq:sp} and \eqref{eq:sm}, we have
\[
\sps = a^2 + b^2 + 2 a b \cos(\theta)
 \quad \text{and} \quad
\sms = a^2 + b^2 - 2 a b \cos(\theta) ,
\]
so $\min\set{\sps,\sms} = a^2 + b^2 - 2 \abs{a b \cos(\theta)}
 = \sqnorm u + \sqnorm v - 2 \absinner u v$.

\textbf{To prove \eqref{eq:theta}}, note that its RHS does not depend on $\lambda$, so it suffices to show that,
for given $a$, $b$, and $\theta$,
\begin{align}\label{eq:smax}
 \max_{\lambda \in \intmp} \min\set{\sps,\sms} & \leq  a^2 + b^2 - a^2 b^2 \cos^2(\theta) .
\end{align}
The maximum must occur for $\lambda$ an endpoint ($-1$ or $+1$), a local maximum of $\sps$ or $\sms$, or where $\sps=\sms$.

If $\lambda=1$ then $\sps=a^2$, while if $\lambda=-1$ then $\sms=a^2$.
In either case \eqref{eq:smax} holds, because $b^2 \geq a^2 b^2 \cos^2(\theta)$
(with equality only when $a=\norm u = 1$ and $\theta=0$).

There is no local maximum of $\sps$ nor $\sms$, since it is evident from \eqref{eq:sp} and \eqref{eq:sm} that both have second derivative $2b^2 \geq 0$. 

Finally, the two are equal where $0= \sps-\sms = 4b(a \cos(\theta) - b \lambda)$, 
i.e., $\lambda = \frac a b \cos(\theta)$ just as in \eqref{eq:geq}.
Substituting this into \eqref{eq:sp} and \eqref{eq:sm} yields 
\begin{align*}
 \sps = \sms &=  a^2+b^2 - a^2 \cos^2(\theta)
   \leq a^2+b^2 - a^2 b^2 \cos^2(\theta) ,
\end{align*}
with equality only when $b = \norm v = 1$.
This establishes \eqref{eq:theta}.

\end{proof}

\begin{lemma}
\label{lem:stability}
    Given vectors $v_1,\dotsc,v_d \in \RR^d$ with $\norm{v_i} \leq 1$, and reals $\lambda_1,\dotsc,\lambda_d \in \intmp$,
    there is a choice of signs $\eta_1,\dotsc,\eta_d \in \setmp$ for which
    $\norm{ \sum_{i=1}^d (\lambda_i - \eta_i) v_i} \leq \sqrt{d}$.
    That is, $(v_1,\dotsc,v_d)$ is $d$-approximating.
    Moreover, the following both hold for all $\delta > 0$.
    \begin{enumerate}
        \item If $\abs{\inner{v_i}{v_j}} \geq \delta$ for some $i \neq j$, then $(v_1,\dotsc,v_d)$ is $(d - \delta^2)$-approximating.
        Moreover, if either $\lambda_i=0$ or $\lambda_j=0$ then there are signs $\eta_1,\dotsc,\eta_d\in\setmp$ such that $\sqnorm{ \sum_{i=1}^d (\lambda_i - \eta_i) v_i} \leq {d-\min\set{2 \delta, 1}}$.
        \item If $\abs{\lambda_i} \geq \delta$ for some $i$, then there exist signs $\eta_1,\dotsc,\eta_d\in\setmp$ such that  the inequality $\sqnorm{ \sum_{i=1}^d (\lambda_i - \eta_i) v_i} \leq d-2\delta+\delta^2 \leq d - \delta$ holds.
    \end{enumerate}
\end{lemma}

The ``moreover'' claim in (i) could improve our bound but for simplicity of exposition we do not use it.
Here and in case (ii), the conclusion is only that the target point $\sum_{i=1}^d \lambda_i v_i$ is well-approximated, not every point, because these cases restrict the choice of $\lambda_i$ and thus the target point.

\begin{proof}
  Let $u_k$ be the error vector accrued by step $k$ and $s_k$ its norm; that is
  \[s_k^2 \defined \sqnorm{u_k} = \sqnorm[\Big]{\sum_{i=1}^k (\eta_i - \lambda_i) v_i}.\]
  By \Cref{lem:uv}'s \eqref{eq:delta}, a greedy choice of $\eta_{k+1}$ gives 
  \[s_{k+1}^2 = \sqnorm[\big]{u_k + (\eta_{k+1} - \lambda_{k+1})v_{k+1}}\leq s_k^2 + (1-\lambda_{k+1}^2) \sqnorm{v_{k+1}} \leq s_k^2 + 1 , \]
  immediately implying the main result. 

For (i), reorder to take $v_i, v_j$ as $v_1, v_2$, with angle $\theta$ between them.
Assume without loss of generality that $\lambda_1 \geq 0$, and choose $\eta_1=1$.
This gives error $u_1 = (1-\lambda_1) v_1 = t v_1$
with $0 \leq t \leq 1$.
By \Cref{lem:uv}'s \eqref{eq:theta},
\begin{align}
s_2^2 
  &\leq \sqnorm{t v_1} + \sqnorm{v_2} - \inner{t v_1}{v_2}^2
  = \sqnorm{v_2} + t^2 (\sqnorm{v_1} - \inner{v_1}{v_2}^2) \label{s22a}
  \\& \leq \sqnorm{v_2} + (\sqnorm{v_1} - \inner{v_1}{v_2}^2)
  \leq \sqnorm{v_1} + \sqnorm{v_2} - \delta^2 ,
  \label{s22b}
\end{align}
where the first inequality exploits $\norm{v_2}\leq 1$, thus $\absinner{v_1}{v_2} \leq \norm{v_1}$, so \eqref{s22a} is maximised by $t^2=1$.

For the second part of (i), with the same setup, assume $\lambda_2=0$.
By \Cref{lem:uv}'s \eqref{eq:lzero},
\[s_2^2 
= \sqnorm{t v_1} + \sqnorm{v_2} - 2\absinner{t v_1}{v_2}
= t^2 \sqnorm{v_1} + \sqnorm{v_2} - 2t\absinner{v_1}{v_2}.\]
As this is a convex quadratic in $t$, its maximum for $t\in[0,1]$ occurs at an endpoint.
Therefore
\[s_2^2 \leq \max\set[\big]{\sqnorm{v_2}, \sqnorm{v_1} + \sqnorm{v_2} - 2 \absinner{v_1}{v_2}} \leq 2 - \min\set{2\delta, 1},\]
and the required bound follows.
  
For (ii), reorder to take $v_i$ as $v_1$. 
Without loss of generality assume $\delta \leq \lambda_1 \leq 1$, and take $\eta_1=1$.
Then
\begin{align*}
 s_1^2 &\leq \sqnorm{(\eta_1-\lambda_1)v_1}
   \leq (1-\delta)^2 \sqnorm {v_1} 
   \leq (1-\delta)^2 
   = 1-2 \delta + \delta^2
   .
\end{align*}
Thus, the final squared norm error is at most$d-2\delta+\delta^2$, which by $\norm\delta \leq 1$ is at most
$d-\delta$.
\end{proof}

We also use the following lemma, which is in essence contained in the proof of Beck's Lemma~2.2 in \cite[pages 7--8]{Beck1983-ef}.

\begin{lemma}
\label{lem:beck-elimination}
    Let $V = (v_1,\dotsc,v_n)$ be a sequence of vectors in $\RR^d$, let $k$ be an integer, and let $W\sseq V$ satisfy $\card{W} \geq k + 1 \geq d + 1$.
    If, for all $Y\sseq W$ with $\card{Y} = k$, the set $(V\setminus W) \union Y$ is $r$-approximating, then $V$ is also $r$-approximating.
\end{lemma}

We will deduce \Cref{lem:beck-elimination} from \Cref{lem:covering}, but, for completeness, we also provide a sketch of Beck's argument.

\begin{proof}[Proof sketch.]
If any $\lambda_i$ in \eqref{eq:approx} is $-1$ or $+1$, setting $\eta_i=\lambda_i$ and eliminating the $i$th variable from the system shows that the approximation for the original system is at least as good as that for the smaller one. 
Initialise $Y=W$.
Choose any $Y' \subseteq Y$ with $\card{Y'}=d+1$.
There is a \emph{nontrivial} solution to $\sum_{i \in Y'} \lambda'_i v_i = 0$. Starting small, scale this solution up until the first time some $\lambda_i+\lambda'_i \in \set{-1, +1}$.
Eliminate variable $i$ and update $Y$ to $Y\setminus \set{i}$.
Repeat until $\card{Y}=k$.
\end{proof}

We will consider \Cref{lem:beck-elimination} as a means of ``eliminating'' vectors: if we assume that some set $V$ is not $r$-approximating, and $W\sseq V$ contains at least $k+1$ elements for some $k\geq d$, then we may find a subset of $V$ which is also not $r$-approximating consisting of all of $V\setminus W$ and an (adversarially chosen) subset of $W$ of size $k$.
This allows us to pass from a large set which is not $r$-approximating to a smaller one, while preserving some of its structure.

It is possible that by more careful tracking of parameters in the proof more could be said about which vectors may be eliminated, but we will consider the elimination in \Cref{lem:beck-elimination} as a black box.

The following lemma (which is a slight generalisation of the standard result that a zonotope can be tiled by appropriate parallelotopes) allows us to break down the (potentially complicated) set $Z(V)$ into a union of translates of the zonotopes generated by subsets of $V$ of size $k$.

\begin{lemma}
\label{lem:covering}
    If $V = (v_1,\dotsc,v_n)$ is a sequence of at least $k\geq d$ vectors in $\RR^d$ and $\cA$ is the family of those subsets $X\sseq S(V)$ isometric to $S(W)$ for some $W\sseq V$ with $\card{W} = k$, then
    \begin{equation*}
        Z(V) = \bigcup \set[\big]{\Conv(X) \st X\in \cA}.
    \end{equation*}
    In other words, defining $p+X = \set{p+x \st x \in X}$ for any point $p$ and set $X$,
    $$Z(V) = \bigcup_{\substack{ W \subseteq V \\ \card W=k}} \bigcup_{p \in S(V \setminus W)} (p + Z(W)) . $$
\end{lemma}

\begin{proof}
    It is immediate that $Z(V) \supseteq \bigcup \set[\big]{\Conv(X) \st X\in \cA}$, so it suffices to prove the reverse inclusion.
    By induction, it suffices to consider the case wherein $\card{V} = k + 1$.
    Let $V = (v_1,\dotsc,v_{k+1})$.
    Consider an arbitrary $p \in Z(V)$. By \Cref{CSeqZ}, $p=\sum_{i=1}^{k+1} \lambda_i v_i$ for some values $\lambda_i \in [-1, +1]$.
    As $k\geq d$, there are some nontrivial weights $\beta_1,\dotsc,\beta_{k+1} \in\RR$ such that $\sum_{i=1}^{k+1} \beta_i v_i = 0$.
    Thus we can re-write $p=\sum_{i=1}^{k+1} (\lambda_i + \gamma \beta_i) v_i$ for any constant $\gamma$.
    Thus we may pick $\gamma$ so that $\abs{\lambda_i + \gamma \beta_i} \leq 1$ for all $i$ and $\abs{\lambda_\l + \gamma \beta_\l} = 1$ for some $\l$.
    
    If $\lambda_\l + \gamma \beta_\l = 1$, then $p \in \Conv(X)$, where $X\sseq S(V)$ is the set 
    \[X = \set[\bigg]{ \sum_{i=1}^{k+1} \eta_i v_i \st \eta_i \in \setmp, \eta_\l = 1}.\]
    The case of $\lambda_\l + \gamma \beta_\l = -1$ is similar. This shows that $Z(V) \sseq \bigcup \set[\big]{\Conv(X) \st X\in \cA}$, and so the claim is proved.
\end{proof}

We now deduce \Cref{lem:beck-elimination} from \Cref{lem:covering}, as promised.

\begin{proof}[Proof of \Cref{lem:beck-elimination}]
    Let $p\in Z(V)$ be arbitrary.
    We must find a point of $S(V)$ within distance $\sqrt{r}$ of $p$. 
    Write $p = q + p_W$, where $q \in Z(V \setminus W)$ and $p_W \in Z(W)$.

    \Cref{lem:covering} applied to $W$ gives us a set $Y\sseq W$ with $\abs{Y} = k$ and a point $s\in S(W\setminus Y)$ with $p_W \in s + Z(Y)$.
    Therefore
    \begin{equation*}
        p - s = q + (p_W - s) \in Z(V\setminus W) + Z(Y) = Z((V\setminus W) \union Y).
    \end{equation*}
    We know by assumption that $(V\setminus W) \union Y$ is $r$-approximating, so there is a point $t\in S((V\setminus W) \union Y)$ such that $\norm{(p-s) - t}^2 \leq r$.
    But then $s+t \in S(V)$ and $\norm{p - (s+t)}^2 \leq r$, so $s+t$ is the required point, and the lemma is proved.
\end{proof}

Despite first appearances, \Cref{lem:beck-elimination,lem:covering} are similar statements with similar proofs, each allowing us to reduce from considering a longer sequence of vectors in $\RR^d$ to working with many shorter sequences.
However, while these two lemmas could be combined into a single more general result, we have refrained from doing so in the interest of keeping the statements simple.

We also record the following fact (which may be trivially deduced from the cosine rule) for ease of referencing.

\begin{fact}
\label{fact:cosine}
    For $x,y$ unit vectors in $\RR^d$ and $\delta \geq 0$, $\inner{x}{y} = 1 - \delta$ if and only if $\norm{x - y} = \sqrt{2\delta}$.
\end{fact}

Finally, we will use the following lemma, which roughly states that an approximately orthogonal sequence of vectors can be well-approximated by a genuinely orthogonal sequence of vectors.

\begin{lemma}
\label{lem:close-to-orthonormal}
    If $x_1,\dotsc,x_d\in \RR^d$ are unit vectors with $\abs{\inner{x_i}{x_j}} \leq \delta$ for all $1\leq i < j \leq d$ and some $\delta > 0$, then there is an orthonormal basis $e_1,\dotsc,e_d \in \RR^d$ with $\norm{x_i - e_i} \leq \delta d$ for all $i$.
\end{lemma}

\begin{proof}
    Let $X$ be the square matrix with column $i$ given by the vector $x_i$, and define $G = X^{T} X$ to be the Gram matrix of these vectors.
    Letting $\norm{}_F$ denote the Frobenius norm, we may then compute
    \begin{equation*}
        \norm{G - I}_F^2 = \sum_{i\neq j} \inner{x_i}{x_j}^2 \leq d(d-1)\delta^2 < d^2\delta^2.
    \end{equation*}
    Let $X = P\Sigma Q^T$ be the singular value decomposition of $X$, and let $U = PQ^T$.
    Note that $U$ is orthogonal and so let $e_1,\dotsc,e_d$ be the columns of $U$, which form an orthonormal basis of $\RR^d$.
    Recalling that the Frobenius norm of a matrix is also equal to the square root of the sum of the squared singular values, if we let $\lambda_1,\dotsc,\lambda_d \geq 0$ be the eigenvalues of $G$, then
    \begin{equation*}
        \norm[\big]{X - U}_F^2 = \norm[\big]{\Sigma - I}_F^2 = \sum_{i=1}^d \p[\big]{\sqrt{\lambda_i} - 1}^2 \leq \sum_{i=1}^d \p[\big]{\lambda_i - 1}^2 = \norm{G - I}_F^2,
    \end{equation*}
    where we have used that $(\lambda_i-1)^2 = (\sqrt{\lambda_i} -1)^2(\sqrt{\lambda_i} + 1)^2 \geq (\sqrt{\lambda_i} -1)^2$ as $\lambda_i \geq 0$.
    We thus find that, for all $i$,
    \begin{equation*}
        \norm{x_i - e_i}_2 = \sqrt{\sum_{j=1}^d (X-U)_{j,i}^2} \leq \sqrt{\sum_{j=1}^d \sum_{\l=1}^d (X - U)_{j,\l}^2} = \norm{X - U}_F \leq \norm{G- I}_F < d\delta,
    \end{equation*}
    as required.
\end{proof}


\section{Bounds for \texorpdfstring{$d\geq 3$}{d ≥ 3}}
\label{sec:balancing}

In this section we prove \Cref{thm:balancing}, except that we prove \Cref{lem:dichotomy} in \Cref{sec:stability}.

As sketched in the introduction, the proof runs by splitting into two cases:
either some pair of vectors is $\alpha$-oblique, or none is.

Define the parameter $\epsilon$ in \Cref{thm:balancing}, the obliqueness parameter $\alpha$, and a parameter $\zeta$ to be used later, by
\begin{align} \label{eq:zetavalue}
  \eps &= \epsvalue,
  \qquad \zeta =  \zetavalue, 
  \text{ and }
  \qquad \alpha = \zeta^{1/4} =\alphavalue.  
\end{align}

\subsection*{\textbf{Case 1. No pair of vectors is \texorpdfstring{$\alpha$}{α}-oblique.}}

If no pair of vectors is oblique, then every pair is nearly parallel or nearly orthogonal, and \Cref{lem:cluster} shows that the vectors all cluster into at most $d$ clusters of nearly parallel vectors.

The parity condition excludes $d$ clusters all of odd size, so
at most $d-1$ clusters have odd size. 
This means that all but $d-1$ vectors can be paired up, nearly cancelling one another out, and it turns out that, all together, these pairs can be signed to have small sum.
This is formalised and quantified by \Cref{lem:clusterShort}.

As observed in the introduction, in this case the point 0 can be well approximated, but $Z(V)$ as a whole cannot be.

We now fill in the details.
In both \Cref{lem:cluster,lem:clusterShort}, we have a real parameter $\rho$ which we will, in applications, take equal to $\alpha$.

\begin{lemma}
    \label{lem:cluster}
    For $d \geq 2$ and $0 \leq \rho \leq 1/(2d^2)$, if $V$ consists of unit vectors in $\RR^d$ and has no $\rho$-oblique pair, then it is $\rho$-clustered. That is, it can be partitioned into at most $d$ clusters
    such that any two vectors within a cluster are $\rho$-parallel and any two in different clusters are $\rho$-orthogonal.
\end{lemma}
\begin{proof}
    For $x,y\in V$, we write $x\sim y$ if $\abs{\inner{x}{y}} \geq 1 - \rho$.
    This will be an equivalence relation provided that it is transitive, i.e.\ if $\abs{\inner{x}{y}} \geq 1 - \rho$ and $\abs{\inner{y}{z}} \geq 1 - \rho$ imply that $\abs{\inner{x}{z}} > \rho$.
    Applying \Cref{fact:cosine}, we find that this is implied by $1 - 4\rho > \rho$, which holds by hypothesis..
    The \emph{clusters} of $V$ are thus the equivalence classes of the relation $\sim$.
    
    We show that $V$ has at most $d$ clusters.
    Indeed, assume for contradiction that there were at least $d+1$ clusters, and let $y_1,\dotsc,y_{d+1}$ be representatives of these clusters, so that for all $1\leq i < j \leq d+1$, we have $\abs{\inner{y_i}{y_j}} \leq \rho$.
    In this case, we may apply \Cref{lem:close-to-orthonormal} to produce an orthonormal basis $f_1,\dotsc,f_d$ of $\RR^d$ such that $\norm{y_i - f_i} \leq \rho d$ for all $i\leq d$.
    Using bilinearity of $\inner \cdot \cdot$, we may then bound, for any $1\leq i \leq d$,
    \begin{equation*}
        \abs{\inner{y_{d+1}}{f_i}} \leq \abs{\inner{y_{d+1}}{y_i}} + \abs{\inner{y_{d+1}}{f_i - y_i}} \leq \rho + \rho d < 2 \rho d.
    \end{equation*}
    Thus, $\sum_{i=1}^d \abs{\inner{y_{d+1}}{f_i}} < 2 \rho d^2$.
    However, $\sum_{i=1}^d \abs{\inner{y_{d+1}}{f_i}} \geq \sum_{i=1}^d \inner{y_{d+1}}{f_i}^2 = 1$.
    Since by hypothesis $2 \rho d^2 \leq 1$, this is a contradiction.
    Thus there are at most $d$ clusters.
\end{proof}

\begin{lemma}
\label{lem:clusterShort}
For $n\not\equiv d\mod{2}$, suppose $n$ unit vectors lie in $d$ clusters $Y_1,\dotsc,Y_d$ (some possibly empty),
with the vector pairs within each cluster $\rho$-parallel
and those across clusters $\rho$-orthogonal.
Then there is a signed sum of the vectors with squared length at most $d-1+2\rho d$.
\end{lemma}

\begin{proof}
    Let the clusters be $Y_1,\dotsc,Y_d$ (some of which may be empty).
    Given this clustering, we now construct a small signed sum of $V$.
    Let $Y_i = (y_1^{(i)}, y_2^{(i)}, \dotsc, y_{m_i}^{(i)})$ for each $i$, and assume, after multiplying some vectors by $-1$ if necessary, that every pair $y,y'\in Y_i$ has $\inner{y}{y'} \geq 1 - \rho$.
    We produce the signed sum on $V$ in stages.
    First of all, let 
    \begin{equation*}
        X_S \defined (y_{2j}^{(i)} - y_{2j-1}^{(i)} \st 1\leq i \leq d, \; 1\leq j \leq \floor{m_i/2})
    \end{equation*}
    be a sequence of ``short'' vectors coming from matching up pairs in each $Y_i$, and let
    \begin{equation*}
        X_L \defined (y_{m_i}^{(i)} \st m_i \equiv 1 \mod{2})
    \end{equation*}
    be the sequence of ``long'' vectors not used in $X_S$.

    Note that, as $n\not\equiv d\mod{2}$, the sequence $X_L$ has at most $d-1$ elements and so is contained in a subspace of $\RR^d$ of dimension at most $d-1$.
    Applying the main result of \Cref{lem:stability}, we may thus find a signed sum $x_L$ of $X_L$ with $\norm{x_L}\leq \sqrt{d-1}$.
    Noting that, by \Cref{fact:cosine}, the norm of any element of $X_S$ is at most $\sqrt{2\rho}$, we can apply \Cref{lem:stability,lem:beck-elimination} to this scaled-down instance to find a signed sum $x_S$ of $X_S$ with $\norm{x_S} \leq \sqrt{2\rho} \sqrt d$.
    By \Cref{lem:uv}, there is thus a signed sum $x = x_L \pm x_S$ of $V$ with squared norm
    ${\norm x}^2 \leq (d-1) + 2\rho d$.
\end{proof}

With \Cref{lem:cluster,lem:clusterShort} in hand, we can conclude the proof in this case.
Firstly, as $\alpha < 1/(2d^2)$, we can apply \Cref{lem:cluster} with $\rho = \alpha$ to find that the vectors $v_1,\dotsc,v_n$ are $\alpha$-clustered.
\Cref{lem:clusterShort} then tells us that there is a signed sum $\sum_{i=1}^n \eta_i v_i$ with squared length at most $d - 1 + 2\alpha d < d - \eps$, as required.

\subsection*{\textbf{Case 2. Some pair of vectors \texorpdfstring{$u,w$ is $\alpha$}{u,w is α}-oblique.}}

If instead some $u,w$ are oblique, then, with no parity condition required, $V$ is $(d-\eps)$-approximating.
We show this as follows.
\Cref{lem:dichotomy} establishes the result in the case $n=d+1$; it is proved in \Cref{sec:stability}.
(The case $n=d$ is not needed for the main result, with its parity condition, but anyway follows from \Cref{lem:stability}.)

\begin{lemma}
\label{lem:dichotomy}
Taking $n=d+1$ and with parameters given by \eqref{eq:zetavalue}, 
if $(v_1,\dotsc,v_{d+1})$ contains any $\zeta$-oblique pair,
then it is $(d - \eps)$-approximating.
\end{lemma}

For $n \geq d+2$,
reserve $u$ and $w$ and, by the method implicit in \Cref{lem:beck-elimination}, reduce the rest of $V$ to a set $Y$ of cardinality $d$.
It suffices to show that $\set{u,w} \cup Y$ is $(d-\eps)$-approximating.

As ${u,w}$ are $\alpha$-oblique, they are also $\zeta$-oblique.
If any $x \in Y$ is $\zeta$-oblique to $u$, reserve $u$ and reduce $\set{w} \cup Y$ to a set $X$ of cardinality $d$.
It suffices to show that $\set{u} \cup X$ is $(d-\eps)$-approximating.
This holds by \Cref{lem:dichotomy}, since $X$ contains either $w$ or $x$ (possibly both), 
so $\set{u} \cup X$ contains a $\zeta$-oblique pair and is of cardinality $d+1$.
Thus, we need only consider the case where no $x \in Y$ is $\zeta$-oblique to $u$ and --- by symmetry --- none is $\zeta$-oblique to $w$.

If any two vectors $y_1,y_2 \in Y$ are $\zeta$-oblique, reduce $\set{u,w} \cup Y$ to a set $X$ of cardinality $d+1$. The set $X$ must contain either both of $u,w$ or both of $y_1,y_2$, thus has an $\zeta$-oblique pair and by \Cref{lem:dichotomy} is $(d-\eps)$-approximating.

Thus it suffices to consider the case of $n=d+2$ vectors, consisting of an $\alpha$-oblique pair $u,w$ and $d$ vectors $Y$, no vector in $Y$ $\zeta$-oblique to $u$ nor $w$, and the set $Y$ $\zeta$-clustered.
The rest of the proof has three main steps.
First, we show that no vector in $Y$ is $\zeta$-parallel to either $u$ or $w$, as if $x \in Y$ was $\zeta$-parallel to $u$, then it would be $\zeta$-oblique to $w$.
Then, as we know that every vector in $Y$ is $\zeta$- orthogonal to both $u$ and $w$, we deduce that all vectors in $Y$ have short projection onto the plane $P$ spanned by $u$ and $w$.
Finally, we use this information and \Cref{lem:stability} to deduce that $\set{u,w} \cup Y$ is $(d-\eps)$-approximating.

    \medskip
    
    \textbf{Step 1. No vector in $Y$ is almost parallel to $u$ or $w$.}
    We now prove that, for all $y\in Y$, we have $\abs{\inner{w}{y}} \leq \zeta$, and similarly for $u$.
    Assume for contradiction that some $y\in Y$ has $\inner{w}{y} > 1 - \zeta$ (replacing $y$ by $-y$ if necessary).
    The idea here is that, as $\zeta$ is much smaller than $\alpha$ and $y$ is very close to $w$, the $\alpha$-obliqueness between $w$ and $u$ approximately carries over to $y$ and $u$,
    easily implying $\zeta$-obliqueness.
    
    By \Cref{fact:cosine}, we find that $\norm{w - y} \leq \sqrt{2\zeta}$, while $\alpha$-obliqueness of $u$ and $w$ implies $\norm{u - w} \in (\sqrt{2\alpha}, \sqrt{2(1-\alpha)})$.
    Thus, replacing $u$ by $-u$ if necessary,
    \begin{equation}
        \label{eq:u-and-y-almost-oblique}
        \norm{u - y} \in \p[\Big]{\sqrt{2\alpha} - \sqrt{2\zeta}, \sqrt{2(1-\alpha)} + \sqrt{2\zeta}} \sseq \p[\big]{\sqrt{2\zeta}, \sqrt{2(1-\zeta)}}
    \end{equation}
    the inclusion following from \eqref{eq:zetavalue} and some simple calculations.
    Applying \Cref{fact:cosine} again, we deduce that $\inner{u}{y} \in (\zeta, 1-\zeta)$, i.e.\ $u$ and $y$ are $\zeta$-oblique.
    This contradicts our assumption that no vector in $Y$ is $\zeta$-oblique to $u$ or $w$, and so we deduce that $\absinner{w}{y}\not> 1-\zeta$.
    We can therefore conclude that, for all $y\in Y$, $\abs{\inner{y}{w}} \leq \zeta$, and similarly $\abs{\inner{y}{u}} \leq \zeta$.

    \medskip

    \textbf{Step 2. All projections of $Y$ onto $P$ are small.}
    Recall that $P$ is the plane through 0 containing the (non-parallel) vectors $u$ and $w$.
    Let $y\in Y$ be arbitrary, and let $z$ be the projection of $y$ onto $P$.
    Thus $z = \beta u + \gamma w$ for some reals $\beta,\gamma$, which we assume are positive (the cases wherein one or both are negative are entirely similar).
    By the triangle inequality, we find
    \[\norm{z} \leq \norm{\beta u} + \norm{\gamma w} = \beta + \gamma.\]
    By expanding the definition of $z$, we know that $\inner{y}{u} = \inner{z}{u} = \beta + \gamma\inner{u}{w} \leq \zeta$ (from step~1), and similarly $\inner{y}{w} = \beta\inner{u}{w} + \gamma \leq \zeta$.
    Adding these two equations together gives
    \[\inner{y}{u} + \inner{y}{w} = (\beta + \gamma)(1 + \inner{u}{w}) \leq 2\zeta.\]
    Thus, as $u,w$ are $\alpha$-oblique, so $1+\inner u w \geq 1-\abs{\inner u w} \geq 1-(1-\alpha)=\alpha$, we find that
    \begin{equation} 
    \label{eq:projection_bound}
        \norm{z} \leq \beta + \gamma = \frac{\inner{y}{u} + \inner{y}{w}}{1 + \inner{u}{w}} \leq \frac{2\zeta}{\alpha} = 2\zeta^{3/4},
    \end{equation}
    and all projections onto $P$ are indeed small, as required.

    \medskip

    \textbf{Step 3. $X$ is $(d-\eps)$-approximating.}
    Write $Y = (y_1,\dotsc,y_d)$.
    We must approximate a vector $p = \lambda_{d+1} u + \lambda_{d+2} w + \sum_{i=1}^d \lambda_i y_i$ by a signed sum of the vectors in $X = Y \union \set{u,w}$.
    Let $P^\perp$ be the $(d-2)$-dimensional space orthogonal to $P$.
    Note that, from steps 1 and 2, we know that the vectors in $Y$ are almost contained in $P^\perp$; it is thus natural to approximate the components of $p$ in $P$ and $P^\perp$ separately.
    Indeed, we use the set $Y$ to approximate the projection of $p$ to $P^\perp$, and $u,w$ to approximate the projection of $p$ to $P$, and then account for the error terms.

    Let $p = p_1 + p_2$, where $p_1\in P$ and $p_2 \in P^\perp$.
    We know from the main result of \Cref{lem:stability} that $Y$ projected down to $P^\perp$ is $(d-2)$-approximating, and we give signs to $Y$ using this approximation.
    To be precise, let $y_i'$ be the projection of $y_i$ into $P^\perp$.
    We know from \eqref{eq:projection_bound} that 
    \begin{equation}
    \label{eq:y-minus-y-prime-bound}
        \norm{y_i - y_i'} \leq 2\zeta^{3/4}.
    \end{equation}
    The main result of \Cref{lem:stability} along with \Cref{lem:beck-elimination} tells us that there are signs $\eta_1,\dotsc,\eta_d$ such that
    \begin{equation}
    \label{eq:p2_bound}
        \norm{\eta_1 y_1' + \dotsb + \eta_d y_d' - p_2}^2 \leq d-2.
    \end{equation}    
    In $P$, \Cref{lem:stability} tells us that $(u,w)$ is $(2-\zeta^{1/2})$-approximating, as $\abs{\inner{u}{w}} > \zeta^{1/4}$.
    We would like to use $u$ and $w$ to approximate the point
    \[p_1 = \lambda_{d+1} u + \lambda_{d+2} w + \sum_{i=1}^d \lambda_i (y_i - y_i').\]
    However, this point does not necessarily lie in the convex hull of $\set{\eta u + \eta'w \st \eta,\eta'\in \setmp}$, and so our approximation result \Cref{lem:stability} need not apply to it.
    Instead, we approximate just $\lambda_{d+1}u + \lambda_{d+2}w$, ignoring for now the $y_i-y_i'$ error terms: we deduce that there are signs $\eta_{d+1}$ and $\eta_{d+2}$ such that
    \begin{equation}
    \label{eq:p1_bound}
        \norm[\Big]{\eta_{d+1} u + \eta_{d+2} w - \p[\Big]{p_1 - \sum_{i=1}^d \lambda_i (y_i - y_i')}}^2 \leq 2 - \zeta^{1/2}.
    \end{equation}
    Combining approximations \eqref{eq:p2_bound} and \eqref{eq:p1_bound}, we may find by splitting the norm into components in $P$ and in $P^\perp$ that
    \begin{equation*}
        \norm[\Big]{\eta_{d+1} u + \eta_{d+2} w + \sum_{i=1}^d \eta_i y_i - p}^2
        = \norm[\Big]{\eta_{d+1} u + \eta_{d+2} w + \sum_{i=1}^d \eta_i (y_i - y_i') - p_1}^2 + \norm[\Big]{\sum_{i=1}^d \eta_i y_i' - p_2}^2.
    \end{equation*}
    The second term on the right-hand side above is bounded above by $d-2$ due to \eqref{eq:p2_bound}.
    By the triangle inequality, and using that $\abs{\eta_i-\lambda_i}\leq 2$, the first term is at most
    \begin{equation}
    \label{eq:final-approximation-bound}
        \p[\bigg]{\norm[\Big]{\eta_{d+1}u + \eta_{d+2} w - \p[\Big]{p_1 - \sum_{i=1}^d \lambda_i (y_i - y_i')}} + 2\sum_{i=1}^d \norm[\big]{y_i - y_i'}}^2 \leq \p[\Big]{\sqrt{2 - \zeta^{1/2}} + 4d\zeta^{3/4}}^2,
    \end{equation}
    where we have used \eqref{eq:p1_bound} and \eqref{eq:y-minus-y-prime-bound} for the two terms inside the square.
    All in all, the total square error of our approximation is at most
    \begin{align}
    d - 2 + \p[\Big]{4d\zeta^{3/4} + \sqrt{2 - \zeta^{1/2}}}^2 
     &= d - \p[\Big]{\zeta^{1/2}- 16d^2\zeta^{3/2} - 8d\zeta^{3/4}\sqrt{2 - \zeta^{1/2}}} \label{eq:step3}
     \\&  \leq d-\frac14 \zeta^{1/2} 
     \; \leq \;  d-\eps.   \notag
    \end{align}
    The first inequality follows from the fact that, in the parentheses in \eqref{eq:step3}, the second term is at most a quarter the first, and the third term (indeed, $8d\zeta^{3/4}\sqrt{2}$) is at most half the first.
    This is immediate from the parameter definitions in \eqref{eq:zetavalue}, as is the final inequality.

    \vspace{1em}

This completes the proof of \Cref{thm:balancing}, except for proving \Cref{lem:dichotomy}.


\section{Stability of approximations}
\label{sec:stability}

The goal of this section is to give a proof of \Cref{lem:dichotomy}.
This proof is somewhat technical, but the ideas involved are not too complex.
We recall that we are working with a sequence $V$ of unit vectors $v_1,\dotsc,v_{d+1}\in \RR^d$ under the assumption that for some distinct $i$ and $j$, $v_i,v_j$ are $\zeta$-oblique, i.e.\ $\zeta < \abs{\inner{v_i}{v_j}} < 1 - \zeta$, and we wish to prove that the sequence of vectors is $(d - \eps)$-approximating.
We now briefly outline the proof, and then present the details.

\begin{enumerate}[label=\arabic*.]
    \item If every subset of $d$ of the $d+1$ vectors contains a $\zeta$-oblique pair, then we are done, so assume that some $d$ of the vectors are $\zeta$-clustered. In fact, we show they are $\zeta$-orthogonal.
    \item We replace the $d$ approximately orthonormal vectors found in the first step with an orthonormal basis $E = (e_1,\dotsc,e_d)$ of $\RR^d$, but in doing so we produce an error term which must be accounted for.
    \item If $y$ is the vector which was not part of the approximately orthogonal set, then $y$ is not close to any $e_i$.
    \item By \Cref{lem:covering}, it suffices to prove that the convex hulls of the parallelotopes within $S(V)$ are sufficiently well approximated.
    Any $d$-subsequence including $y$ is suitably approximating, and so it remains to approximate the two hypercubes $Z(E) + y$ and $Z(E) - y$.
    \item Every point of $Z(E) + y$ is suitably well approximated by either some point of $S(E) + y$ or a specific point of $S(E) - y$.
\end{enumerate}

\begin{proof}[Proof of \Cref{lem:dichotomy}]
    We follow the proof outline presented above.

    \textbf{Step 1: finding an approximately orthogonal $d$-subsequence.}
    Assume for contradiction that, amongst every $d$-subsequence of the $d+1$ vectors, there is some pair $u,v$ with $\abs{\inner{u}{v}} > \eps ^{1/2}$.
    Then, by \Cref{lem:stability}, we find that this $d$-subsequence is $(d-\eps)$-approximating, and so we are done by \Cref{lem:beck-elimination}.

    Thus we may write our sequence of $d+1$ vectors as $(x_1,\dotsc,x_d,y)$, where for all $1\leq i < j \leq d$, we have $\abs{\inner{x_i}{x_j}} \leq \eps^{1/2}$.
    Recalling that $\zeta = \zetavalue$ and $\eps = \epsvalue$, each pair $x_i$ and $x_j$ is $\zeta$-orthogonal.
    
    \vspace{1em}

    \textbf{Step 2: moving to an orthonormal basis.}
    We now replace $x_1,\dotsc,x_d$ with an orthonormal basis, at the cost of accumulating an error term which we must account for.
    Let $e_1,\dotsc,e_d$ be the orthonormal basis guaranteed by \Cref{lem:close-to-orthonormal}, where for all $1\leq i \leq d$ we have
    \begin{equation}
    \label{eq:orthonormal-approximation}
        \norm{x_i - e_i} \leq \eps^{1/2} d.
    \end{equation}
    We claim that, instead of proving that $(x_1,\dotsc,x_d,y)$ is $(d-\eps)$-approximating, it suffices to prove that the sequence $W = (e_1,\dotsc,e_d,y)$ is $(d - 6 \eps^{1/4} d^3)$-approximating.

    Indeed, if $W$ is $(d - 6 \eps^{1/4} d^3)$-approximating, then for any point $p$ in $Z(W)$, let $\lambda_1,\dotsc,\lambda_d,\lambda_{d+1}$ be such that $p = \sum_{i=1}^d \lambda_i e_i + \lambda_{d+1} y$. Define the point $p'$ to be $\sum_{i=1}^d \lambda_i x_i + \lambda_{d+1} y$.
    The triangle inequality thus tells us that
    \begin{align*}
        \norm[\Big]{\sum_{i=1}^d \eta_i x_i + \eta_{d+1} y - p'}
        &\leq \norm[\Big]{\sum_{i=1}^d \eta_i e_i + \eta_{d+1} y - p} + \norm[\Big]{\sum_{i=1}^d (\eta_i - \lambda_i) (x_i - e_i)} \\
        &\leq \sqrt{d - 6 \eps^{1/4} d^3} + 2 \eps^{1/2} d^2.
    \end{align*}
    It suffices that the above is at most $\sqrt{d-\eps}$.
    Rearranging, $\sqrt{d - 6 \eps^{1/4} d^3} + 2 \eps^{1/2} d^2 \leq \sqrt{d - \eps}$ is equivalent to
    \begin{equation*}
        2\eps^{1/2}d^2 \leq \frac{6\eps^{1/4}d^3 - \eps}{\sqrt{d-\eps} + \sqrt{d - 6\eps^{1/4}d^3}},
    \end{equation*}
    and we may compute
    \begin{equation*}
        \frac{6\eps^{1/4}d^3 - \eps}{\sqrt{d-\eps} + \sqrt{d - 6\eps^{1/4}d^3}} > \frac{5\eps^{1/4}d^3}{2\sqrt{d}} > 2\eps^{1/4}d^{5/2} > 2 \eps^{1/2}d^2,
    \end{equation*}
    as required.
    We therefore now work to prove that $W$ is $(d - 6 \eps^{1/4} d^3)$-approximating.

    \vspace{1em}

    \textbf{Step 3: the vectors $y$ and $e_i$ are not close.}
    As $E=(e_1,\dotsc,e_d)$ is an orthonormal basis, we may write $y = \sum_{i=1}^d y_i e_i$ for some reals $y_i$ with $\sum_{i=1}^d y_i^2 = 1$.
    By replacing $e_i$ and $x_i$ with $-e_i$ and $-x_i$ if necessary, we may moreover assume that, for all $i$, $y_i \geq 0$.
    We know by assumption that there is a $\zeta$-oblique pair and $\absinner{x_i}{x_j} \leq \eps^{1/2} < \zeta$; the $\zeta$-oblique pair must therefore contain $y$. 
    Thus assume that, for some $k$, we have $\zeta \leq \abs{\inner{y}{x_k}} \leq 1-\zeta$.
    We have that
    \begin{equation*}
        y_k
        \defined \inner{y}{e_k} 
        = \inner{y}{x_k} + \inner{y}{e_k - x_k},
    \end{equation*}
    and recalling from \eqref{eq:orthonormal-approximation} that $\norm{e_i - x_i} \leq \eps^{1/2} d$ for all $i$, and $y$ is a unit, we find that
    \begin{equation}
    \label{eq:yk_bound}
        \zeta / 2 < \zeta - \eps^{1/2} d \leq y_k \leq 1 - \zeta + \eps^{1/2} d < 1 - \zeta/2.
    \end{equation}
    Therefore, for all $j\neq k$, the fact that $\inner{y}{e_j}^2 + \inner{y}{e_k}^2 \leq 1$ implies that
    \begin{equation}
    \label{eq:y1_bound}
        y_j^2 \leq 1 - \zeta^2/4.
    \end{equation}
    Finally, \eqref{eq:yk_bound} implies that \eqref{eq:y1_bound} holds for $j=k$ as well.
    Inequality \eqref{eq:yk_bound} precisely says that $y$ is not close to any basis vector $e_j$.

    \vspace{1em}

    \textbf{Step 4: approximation by $d$-subsequences including $y$.}
    By \Cref{lem:covering}, it suffices to consider approximations in the convex hulls of the signed sums of $d$-subsequences of $W$.
    If $Y\sseq W$ is a $d$-subsequence with $y\in Y$, then we will show that $Y$ is suitably approximating, even if we ignore the vector not in $Y$.
    We will use \Cref{lem:stability}, and to this end compute a lower bound for $\max\set{y_2,y_3,\dotsc,y_d}$ (recalling that these numbers are assumed to be non-negative).
    By symmetry this lower bound will also hold for any other $d-1$ of the coefficients $y_1,\dotsc,y_d$.
    Indeed, recalling inequality \eqref{eq:y1_bound}, we find that
    \begin{equation*}
        \sum_{i=2}^d y_i^2 = 1 - y_1^2 \geq \frac{\zeta^2}{4}.
    \end{equation*}
    It is therefore clear that $\max\set{y_2,y_3,\dotsc,y_d} > \zeta/\sqrt{4d}$.
    Thus, by symmetry, for any $d$-subsequence $Y\sseq W$ as described above, there is some $i$ such that $e_i \in Y$ and $\inner{y}{e_i} > \zeta/\sqrt{4d}$, and so by \Cref{lem:stability}, $Y$ is $(d - \zeta^2/(4d))$-approximating.

    Recalling that $\zeta = \zetavalue$, we find that $\zeta^2 / (4d) > 6\eps^{1/4} d^3$, and so $W$ can $(d-6\eps^{1/4} d^3)$-approximate any point in the convex hull of a parallelotope corresponding to a $d$-subsequence $Y$ as above.

    \vspace{1em}

    \textbf{Step 5: approximating the centre of a hypercube.}
    The remaining part of the proof of \Cref{lem:dichotomy} is to consider the $d$-subsequences of $W$ not including $y$ --- that is, the sequence $E = (e_1,\dotsc,e_d)$.
    The set $Z(E)$ is a hypercube, and points close to the centre of this hypercube are not well-approximated by the vertices, and so we will have to make use of the other points of $S(W)$.
    Moving the origin by $y$, we must show that every point in $Z(E)$ is $(d-6\eps^{1/4} d^3)$-approximated by a point of $S(E) \union (S(E) + 2y)$.
    In fact, the only point of $S(E) + 2y$ we will use is $\sum_{i=1}^d (2y_i - 1) e_i$.

    By \Cref{lem:stability}, the point $p = \sum_{i=1}^d \lambda_i e_i$ is $(d-6\eps^{1/4} d^3)$-approximated by a point of $S(E)$ unless, for all $i$, $\abs{\lambda_i} < 6\eps^{1/4} d^3$, i.e.\ $p$ is very close to the centre of the hypercube $Z(E)$.

    Thus assume that we do have $\abs{\lambda_i} < 6\eps^{1/4} d^3$ for every $i$.
    In this case, $\norm{p}^2 < 36\eps^{1/2} d^7$, and it suffices to prove that any such point $p$ is within distance $\sqrt{d - 6\eps^{1/4} d^3}$ of $\sum_{i=1}^d (2y_i - 1)e_i$.
    In particular, it suffices by the triangle inequality to prove that
    \begin{equation*}
        \sum_{i=1}^d (2y_i - 1)^2 \leq \p[\Big]{\sqrt{d - 6\eps^{1/4} d^3} - \sqrt{36 \eps^{1/2} d^7}}^2,
    \end{equation*}
    which, after some rearranging, is equivalent to
    \begin{equation*}
        4 + 6\eps^{1/4} d^3 - 36 \eps^{1/2} d^7 + 12\eps^{1/4} d^{7/2} \sqrt{d - 6\eps^{1/4} d^3} \leq 4 \sum_{i=1}^d y_i,
    \end{equation*}
    and so is implied by
    \begin{equation}
    \label{eq:y-target}
        \sum_{i=1}^d y_i \geq 1 + \frac{9}{2} \eps^{1/4} d^4.
    \end{equation}
    To prove this, recall that $\zeta/2 < y_k < 1-\zeta/2$ and note that
    \begin{equation*}
        -y_k + \sum_{i=1}^d y_i \geq -y_k^2 + \sum_{i=1}^d y_i ^2 = 1 - y_k^2 = 1 - y_k + (y_k - y_k^2),
    \end{equation*}
    and so
    \begin{equation*}
        \sum_{i=1}^d y_i \geq 1 + y_k - y_k^2 \geq 1 + \frac{\zeta}{2}\p[\Big]{1 - \frac{\zeta}{2}} > 1 + \frac{\zeta}{4}
    \end{equation*}
    as $y_k \in (\zeta/2, 1 - \zeta/2)$.
    Finally, recalling that $\zeta = \zetavalue$, we may deduce \eqref{eq:y-target}, completing the proof of \Cref{lem:dichotomy}.
\end{proof}


\section{Concluding remarks and open problems}
\label{sec:conclusion}

We have demonstrated that, for any dimension $d$, the infimum of the probability $\prob{\norm{\sum_{i=1}^n \xi_i v_i} \leq r}$ over choices of unit vectors $v_1,\dotsc,v_n$ in $\RR^d$ is sensitive to the parity of $n$, particularly for $r$ just below $\sqrt{d}$. 
However, our results are only a first step in this direction, and there is much more to be understood.
Perhaps the most prominent open problem is \Cref{q:balancing}, which we repeat here.

\balanceQuestionMain*

\Cref{thm:balancing} gives a weaker version of the above, showing only that we can find signs for which $\norm[\big]{\sum_{i=1}^n \eta_i v_i} \leq \sqrt{d-\eps}$ for some $\eps = \eps(d) > 0$.
It is known, due to the results displayed in \Cref{table}, that \Cref{q:balancing} has a positive answer in dimension $d=2$, but even the case of $d=3$ remains unknown.
Alongside the case of $d=3$, another appealing starting point is considering \Cref{q:balancing} in the case of $n=d+1$.

If the answer to \Cref{q:balancing} is yes, then we repeat the question from \cite{HPS26} of whether there is also a double-jump phase transition in dimension $d\geq 3$.
More precisely, we ask the following question.

\begin{question}
    Is it the case that, for all $d\geq 3$, all $r$ with $\sqrt{d-1} < r \leq \sqrt{d}$, all $n\not\equiv d\mod{2}$, we have
    \begin{equation*}
        \inf_{v_1,\dotsc,v_n \in \SS^{d-1}} \prob[\Big]{\norm[\Big]{\sum_{i=1}^n \xi_i v_i} \leq r} = \Theta_{d,r}(n^{-d/2}) \; ? 
    \end{equation*}
\end{question}

In a different direction, while \Cref{thm:construction} shows that at radius $r=1$ in dimension $d=2$ with $n$ odd, we can see exponential decay in the probability that a signed sum lies in the unit disc, it remains open what the true asymptotic behaviour is.

\begin{question}
    What is the correct base of the exponent for the minimal value of $\prob{\norm{\sum_{i=1}^n \xi_i v_i} \leq 1}$ over $v_1,\dotsc,v_n\in\SS^1$ for $n$ odd? In other words, does the limit
    \begin{equation*}
        \lim_{\substack{n\to\infty \\ n\text{ odd}}} \p[\bigg]{\inf_{v_1,\dotsc,v_n \in \SS^1} \prob[\Big]{\norm[\Big]{\sum_{i=1}^n \xi_i v_i} \leq 1}^{1/n}}
    \end{equation*}
    exist, and if so, what is its value?
\end{question}

We remark that the construction in \Cref{thm:construction} also gives decay of $\prob[\big]{\norm[\big]{\sum_{i=1}^n \xi_i v_i} \leq \sqrt{d-1}}$ of the form $2^{-n/2+O(1)}$ in $\RR^d$: let $e_1,\dotsc,e_d$ be an orthonormal basis of $\RR^d$, put the construction from \Cref{thm:construction} in the plane spanned by $e_1$ and $e_2$, and then also include $e_3,\dotsc,e_d$ in the set of vectors.
However, it would be interesting to know whether faster exponential decay can be achieved in $d$ dimensions.


\section{Acknowledgements}

The first author was funded by the Internal Graduate Studentship of Trinity College, Cambridge.


\bibliographystyle{abbrvnat}  
\renewcommand{\bibname}{Bibliography}
\bibliography{main}


\end{document}

%% file: header.tex
\usepackage{amsmath,amssymb,amsthm}

\usepackage{mathrsfs}

\usepackage{mathtools}
\usepackage{commath}

\usepackage{thmtools}
\usepackage{thm-restate}

\usepackage{cases}
\usepackage{enumitem}
\setlist[enumerate,1]{label=(\roman*)}

\usepackage[pdftex, pdfborderstyle={/S/U/W 0}]{hyperref}
\hypersetup{
    colorlinks=true,
    linkcolor=magenta,
    citecolor=cyan,
}

\usepackage{cleveref}

\usepackage{etoolbox}

\usepackage{comment}

\usepackage[numbers]{natbib}

\numberwithin{equation}{section}

\ifdefined\thmcolor
\declaretheoremstyle[
  shaded={bgcolor=\thmcolor}
]{plain}
\else
\fi

\ifdefined\defcolor
\declaretheoremstyle[
  headfont=\normalfont\bfseries,
  bodyfont=\normalfont,
  shaded={bgcolor=\defcolor}
]{noital}
\else
\declaretheoremstyle[
  headfont=\normalfont\bfseries,
  bodyfont=\normalfont,
]{noital}
\fi

\declaretheorem[style=plain,numberwithin=section,name=Theorem]{theorem}

\declaretheorem[style=plain,sibling=theorem,name=Lemma]{lemma}

\declaretheorem[style=plain,sibling=theorem,name=Conjecture]{conjecture}
\declaretheorem[style=plain,sibling=theorem,name=Claim]{claim}
\declaretheorem[style=plain,sibling=theorem,name=Fact]{fact}

\declaretheorem[style=plain,sibling=theorem,name=Question]{question}

\declaretheorem[style=plain,numbered=no,name=Theorem]{theorem-n}
\declaretheorem[style=plain,numbered=no,name=Proposition]{proposition-n}
\declaretheorem[style=plain,numbered=no,name=Lemma]{lemma-n}
\declaretheorem[style=plain,numbered=no,name=Corollary]{corollary-n}
\declaretheorem[style=plain,numbered=no,name=Conjecture]{conjecture-n}
\declaretheorem[style=plain,numbered=no,name=Claim]{claim-n}
\declaretheorem[style=plain,numbered=no,name=Fact]{fact-n}
\declaretheorem[style=plain,numbered=no,name=Open Problem]{openproblem-n}
\declaretheorem[style=plain,numbered=no,name=Question]{question-n}
\declaretheorem[style=plain,numbered=no,name=Observation]{observation-n}

\declaretheorem[style=noital,sibling=theorem,name=Remark]{remark}
\declaretheorem[style=noital,sibling=theorem,name=Definition]{definition}

\declaretheorem[style=noital,numbered=no,name=Remark]{remark-n}
\declaretheorem[style=noital,numbered=no,name=Definition]{definition-n}
\declaretheorem[style=noital,numbered=no,name=Construction]{construction-n}
\declaretheorem[style=noital,numbered=no,name=Example]{example-n}

\newcommand{\defined}{\mathrel{\coloneqq}}

\DeclarePairedDelimiter{\p}{\lparen}{\rparen}

\newcommand{\st}{\mathbin{\colon}}

\undef{\set}
\DeclarePairedDelimiter{\set}{\lbrace}{\rbrace}

\undef{\emptyset}
\newcommand{\emptyset}{\varnothing}

\DeclarePairedDelimiter{\card}{\lvert}{\rvert}

\newcommand{\union}{\mathbin{\cup}}

\DeclarePairedDelimiter{\floor}{\lfloor}{\rfloor}
\DeclarePairedDelimiter{\ceil}{\lceil}{\rceil}

\undef{\mod}
\newcommand{\mod}[1]{\ (\mathrm{mod}\ #1)}

\undef{\abs}
\DeclarePairedDelimiterX{\abs}[1]
  {\lvert}{\rvert}{\ifblank{#1}{\,\cdot\,}{#1}}

\undef{\norm}
\DeclarePairedDelimiterX{\norm}[1]
  {\lVert}{\rVert}{\ifblank{#1}{\,\cdot\,}{#1}}

\DeclarePairedDelimiterX{\inner}[2]
  {\langle}{\rangle}{\ifblank{#1}{\,\cdot\,}{#1},\ifblank{#2}{\,\cdot\,}{#2}}
  
\DeclarePairedDelimiterX{\absinner}[2]
  {|\langle}{\rangle|}{\ifblank{#1}{\,\cdot\,}{#1},\ifblank{#2}{\,\cdot\,}{#2}}

\DeclareMathDelimiter{\given}
  {\mathbin}{symbols}{"6A}{largesymbols}{"0C}

\DeclareMathOperator{\Prob}{\mathbb{P}}
\DeclarePairedDelimiterXPP{\prob}[1]
  {\Prob}{\lparen}{\rparen}{}
  {\renewcommand{\given}{\nonscript\;\delimsize\vert\nonscript\;\mathopen{}}#1}

\DeclareMathOperator{\Expec}{\mathbb{E}}
\DeclarePairedDelimiterXPP{\expec}[1]
  {\Expec}{\lparen}{\rparen}{}
  {\renewcommand{\given}{\nonscript\;\delimsize\vert\nonscript\;\mathopen{}}#1}

\DeclareMathOperator{\Var}{Var}
\DeclarePairedDelimiterXPP{\var}[1]
  {\Var}{\lparen}{\rparen}{}
  {\renewcommand{\given}{\nonscript\;\delimsize\vert\nonscript\;\mathopen{}}#1}

\DeclareMathOperator{\Cov}{Cov}
\DeclarePairedDelimiterXPP{\cov}[2]
  {\Cov}{\lparen}{\rparen}{}{#1,#2}

\newcommand{\eps}{\varepsilon}
\newcommand{\sseq}{\subseteq}

\let\l\relax
\newcommand{\l}{\ell}

\let\SS\relax

\newcommand{\RR}{\mathbb{R}}
\newcommand{\SS}{\mathbb{S}}

\newcommand{\cA}{\mathcal{A}}

